\documentclass[a4paper,11pt]{for_arXiv}
\usepackage{amsmath,amsfonts,amssymb}

\newtheorem{example}[theorem]{Example}

\newcommand{\al}{\mathbb{\alpha}}
\newcommand{\diam}{\diamondsuit_{\alpha}}

\newcommand{\N}{\mathbb{N}}

\newcommand{\R}{\mathbb{R}}
\newcommand{\rar}{\mbox{$\rightarrow$}}

\newcommand{\T}{\mathbb{T}}
\newcommand{\tr}{\triangle}

\newcommand{\Z}{\mathbb{Z}}

\begin{document}


\title{Diamond-alpha Polynomial Series on Time
Scales\thanks{Accepted to the 8th Portuguese Conference 
on Automatic Control -- CONTROLO'2008, 21 to 23 July 2008, 
UTAD University, Vila Real, Portugal.}}

\author{\textbf{Dorota Mozyrska$^1$ and Delfim F. M. Torres$^2$}}

\date{$^1$Faculty of Computer Science \\ Department of Mathematics\\
Bia{\l}ystok Technical University\\
Wiejska 45A, Bia{\l}ystok 15-351, Poland\\
admoz@w.tkb.pl\\[0.3cm]
$^2$Department of Mathematics\\
University of Aveiro\\
3810-193 Aveiro, Portugal\\
delfim@ua.pt}

\maketitle


\begin{abstract}
\noindent \emph{The objective of this paper is twofold: (i) to survey existing
results of generalized polynomials on time scales, covering
definitions and properties for both delta and nabla derivatives;
(ii) to extend previous results by using the more general notion
of diamond-alpha derivative on time scales. We introduce a new
notion of combined-polynomial series on a time scale, as a convex
linear combination of delta and nabla generalized series. Main
results are formulated for homogenous time scales.  As an example, we compute
diamond-alpha derivatives on time scales for delta and nabla
exponential functions.
}

\medskip


\medskip

\end{abstract}


\section{Introduction}

Polynomial series are of great importance in control theory. Both
continuous and discrete polynomial series are useful in
approximating state and/or control variables, in modal reduction,
optimal control, and system identification, providing effective
and efficient computational methods \cite{Jacobsohn,Para,Sadek}.

From recent years, the theory of control for discrete and
continuous time is being unified and extended by using the
formalism of time scales: see \cite{Zbig07,Zbig06} and references
therein. Looking to the literature on time scales, one understands
that such unification and extension is not unique. Two main
directions are being followed: one uses $\Delta$-derivatives while
the other chooses $\nabla$-derivatives instead. In this paper we
adopt the more general notion of diamond-$\alpha$ derivative
\cite{SFHD}, and give the first steps on a correspondent theory of
polynomial series. As particular cases, for $\alpha=1$ we get
$\Delta$-polynomial series; when $\alpha=0$ we obtain
$\nabla$-series. By choosing the time scale to be the real
(integer) numbers, we obtain the classical continuous (discrete)
polynomial series.

Diamond-alpha derivatives have shown in computational experiments
to provide efficient and balanced approximation formulas, leading
to the design of more reliable numerical methods
\cite{Sheng1,SFHD}. We claim that the combined dynamic polynomial
series here introduced are useful in control applications.



\section{The $\Delta$, $\nabla$, and $\diam$ calculus}

Here we give only very short introduction (with basic definitions)
on three types of calculus on time scales. For more information we
refer the reader to \cite{AB,And,Bh,Ozkan,Sheng,SFHD}.

By a time scale, here denoted by $\T$, we mean a nonempty closed
subset of $\R$. As the theory of time scales give a way to unify
continuous and discrete analysis, the standard cases of time
scales are $\T=\R$, $\T=\Z$, or $\T=c\Z$, $c>0$.

For $t\in\T$, the forward jump operator $\sigma$ and the
graininess function $\mu$ are defined by
$\sigma(t)=\inf\{s\in\T:s>t\}$ and $\sigma(\sup\T)=\sup\T$ if
$\sup\T<+\infty$; $\mu(t)=\sigma(t)-t$.  Moreover, we have the backward
operator $\rho$ and the backward graininess function $\nu$ defined by $\rho(t)=\sup\{s\in\T: s<t\}$ and
$\rho(\inf\T)=\inf\T$ if $\inf\T>-\infty$; $\nu(t)=t-\rho(t)$.
In the continuous-time
case, \textrm{i.e.} when $\T=\R$, we have $\sigma (t)=\rho(t)=t$
and $\mu(t)=\nu(t)=0$ for all $t\in\R$. In the discrete-time case,
$\sigma(t)=t+1$, $\rho(t)=t-1$, and $\mu(t)=\nu(t)=1$ for each
$t\in\T=\Z$. For the composition between a function
$f:\T\rightarrow\R$ and functions $\sigma:\T\rightarrow\T$ and
$\rho:\T\rightarrow\T$, we use the abbreviations
$f^{\sigma}(t)=f(\sigma(t))$ and $f^{\rho}(t)=f(\rho(t))$.
A point $t$ is called
left-scattered (right-scattered) if $\rho(t)<t,$ ($\sigma(t)>t$). A
point $t$ is called left-dense (right-dense) if
$\rho(t)=t,$ ($\sigma(t)=t$).
 The set
$\T^k$ is defined by $\T^k:=\T\backslash (\rho(\sup\T),\sup\T]$ if
$\sup\T<\infty$, and $\T^k=\T$ if $\sup\T=\infty$; the set $\T_k$
by $\T_k:=\T\backslash[\inf\T,\sigma(\inf\T))$ if $|\inf\T|
<\infty$, and $\T_k=\T$ if $\inf\T=-\infty$. Moreover, $\T^{k^{n+1}}:=\left(\T^{k^n}\right)^k,$
$\T_{k^{n+1}}:=\left(\T_{k^n}\right)_k$ and
$\T^k_k:=\T^k\cap \T_k$.

For a function $f:\T\rightarrow\R$, we define the
$\Delta$-derivative of $f$ at $t\in\T^k$, denoted by
$f^{\Delta}(t)$, to be the number, if it exists, with the property
that for all $\varepsilon>0$, exists a neighborhood $U\subset \T$
of $t\in\T^k$ such that for all $s\in U$,
$|f^\sigma(t)-f(s)-f^{\Delta}(t)(\sigma(t)-s)|\leq
\varepsilon|\sigma(t)-s|$. Function $f$ is said to be
$\Delta$-differentiable on $\T^k$ provided $f^{\Delta}(t)$ exists
for all $t\in\T^k$.

The $\nabla$-derivative of $f$, denoted by $f^{\nabla}(t)$, is
defined in a similar way: it is the number, if it exists, such
that for all $\varepsilon>0$ there is a neighborhood $V\subset \T$
of $t\in\T_k$ such that for all $s\in V,$
$|f^\rho(t)-f(s)-f^{\nabla}(t)(\rho(t)-s)|\leq
\varepsilon|\rho(t)-s|$. Function $f$ is said to be
$\nabla$-differentiable on $\T_k$ provided $f^{\nabla}(t)$ exists
for all $t\in\T_k$.

\begin{example} The classical settings are obtained choosing $\T = \R$
and $\T = \Z$:
\begin{enumerate}
\item Let $\T=\R$. Then, $f^{\Delta}(t)=f^{\nabla}(t)=f'(t)$ and
$f$ is $\Delta$ and $\nabla$ differentiable if and only if it is
differentiable in the ordinary sense.

\item Let $\T=c\Z$, $c>0$. Then,
$f^{\Delta}(t)=\frac{f(t+c)-f(t)}{c}$ and
$f^{\nabla}(t)=\frac{1}{c}\left(f(t)-f(t-c)\right)$ always exist.
\end{enumerate}
\end{example}

It is possible to establish some relationships between $\Delta$
and $\nabla$ derivatives.

\begin{theorem}{\rm \cite{Bh}}
(a) Assume that $f:\T\rightarrow\R$ is delta differentiable on
$\T^k$. Then, $f$ is nabla differentiable at $t$ and
$f^{\nabla}(t)=f^{\Delta}(\rho(t))$ for all $t\in\T_k$ such that
$\sigma(\rho(t))=t$. (b) Assume that $f:\T\rightarrow\R$ is nabla
differentiable on $\T_k$. Then, $f$ is delta differentiable at $t$
and $f^{\Delta}(t)=f^{\nabla}(\sigma(t))$ for all $t\in\T^k$ such
that $\rho(\sigma(t))=t$.
\end{theorem}

A function $f:\T \rar \R$ is called rd-continuous provided it is
continuous at right-dense points in $\T$ and its left-sided limits
exist (finite) at left-dense points in $\T$. The class of real
rd-continuous functions defined on a time scale $\T$ is denoted by
$C_{rd}(\T,\R)$.

If $f\in C_{rd}(\T,\R)$, then there exists a function $F(t)$ such
that $F^{\Delta}(t)=f(t)$. The delta-integral is defined by
$\int_{a}^bf(t)\Delta t=F(b)-F(a)$.

Similarly, a function $f:\T \rar \R$ is called ld-continuous
provided it is continuous at left-dense points in $\T$ and its
right-sided limits exist (finite) at right-dense points in $\T$.
The class of real ld-continuous functions defined on a time scale
$\T$ is denoted by $C_{ld}(\T,\R)$. If $f\in C_{ld}(\T,\R)$, then
there exists a function $G(t)$ such that $G^{\nabla}(t)=f(t)$. In
this case we define $\int_{a}^bf(t)\nabla t=G(b)-G(a)$.

\begin{example}
Let $\T=c\Z$, $c>0$, and $f\in C_{rd}(\T,\R) \cap C_{ld}(\T,\R)$.
Then, one has: $\int_{a}^bf(t)\Delta t=c\sum_{t=a}^{b-c}f(t)$,
$\int_{a}^bf(t)\nabla t=c\sum_{t=a+c}^{b}f(t)$.
\end{example}

\begin{definition}\cite{Sheng1}
Let $\mu_{ts}=\sigma(t)-s$, $\eta_{ts}=\rho(t)-s$, and
$f:\T\rightarrow\R$. The diamond-alpha derivative of $f$ at $t$ is
defined to be the value $f^{\diam}(t)$, if it exists, such that
for all $\varepsilon>0$ there is a neighborhood $U\subset \T$ of
$t$ such that for all $s\in U,$
$$\left|\al\left[f^{\sigma}(t)-f(s)\right]\eta_{ts}
+(1-\al)\left[f^{\rho}(t)-f(s)\right]\mu_{ts}
-f^{\diam}(t)\mu_{ts}\eta_{ts}\right|\leq
\varepsilon|\mu_{ts}\eta_{ts}| \, . $$
We say that function $f$ is
$\diam$-differentiable on $\T^k_k$, provided $f^{\diam}(t)$ exists for all $t\in\T^k_k$.
\end{definition}

\begin{theorem}\cite{Sheng1}\label{thsheng}
Let $f:\T\rightarrow\R$  be simultaneously $\Delta$ and $\nabla$
differentiable at $t\in\T^k_k$. Then, $f$ is $\diam$-differentiable at
$t$ and $f^{\diam}(t)=\al f^{\Delta}(t)+(1-\al)f^{\nabla}(t)$,
$\al\in[0,1]$.
\end{theorem}

\begin{remark}
The $\diam$-derivative is a convex combination of delta and nabla
derivatives. It reduces to the $\Delta$-derivative for $\al=1$ and
to the $\nabla$-derivative for $\al=0$. The case $\al=0.5$ has
proved to be very useful in applications. For more on the theory
of $\diam$-derivatives than that we are able to provide here, we
refer the interested reader to \cite{Ozkan,Sheng1,Sheng,SFHD}.
\end{remark}

The same idea used to define the combined derivative is taken to
define the combined integral.

\begin{definition}\label{diaminteg}
Let $a,b\in\T$ and $f\in C_{rd}(\T,\R)\cap C_{ld}(\T,\R)$. Then,
the $\diam$-integral of $f$ is defined by
$\int_a^bf(\tau)\diam\tau=\al
\int_a^bf(\tau)\Delta\tau+(1-\al)\int_a^bf(\tau)\nabla\tau$, where
$\al\in[0,1]$.
\end{definition}

In general the $\diam$-derivative of $\int_a^t f(\tau)\diam\tau$
with respect to $t$ is not equal to $f(t)$ \cite{Sheng}.

Next proposition gives direct formulas for the $\diam$-derivative
of the exponential functions $e_p(\cdot,t_0)$ and
$\hat{e}_p(\cdot,t_0)$. For the definition of exponential and
trigonometric functions on time scales see, \textrm{e.g.},
\cite{Bh}.

\begin{proposition}
\label{diamexpder} Let $\T$  be a time scale with the following properties: $\sigma(\rho(t))=t$, and $\rho(\sigma(t))=t$. Assume that $t_0\in\T$ and for all $t\in\T$ one has  $1+\mu(t)p(t)\neq 0$, and $1-\nu(t)p(t)\neq 0$. Then,
\begin{gather}
e_p^{\diam}(t,t_0)=\left[\al p(t)+\frac{(1-\al)p^{\rho}(t)}{1+\nu(t)p^{\rho}(t)}\right]e_p(t,t_0) \, , \label{eq:a} \\
\hat{e}_p^{\diam}(t,t_0)=\left[(1-\al)p(t)+
\frac{\al p^{\sigma}(t)}{1-\mu(t)p^{\sigma}(t)}\right]\hat{e}_p(t,t_0) \, , \label{eq:b}
\end{gather}
for $t\in\T^k_k$.
\end{proposition}

\begin{proof}
Firstly, recall that $f^{\sigma}= f+\mu f^{\Delta}$ and $f^{\rho}
= f-\nu f^{\nabla}$. Hence,
$e_p^{\nabla}(t,t_0)=p^{\rho}(t)e_p^{\rho}(t,t_0)
=p^{\rho}(t)\left(e_p(t,t_0)-\nu(t)e_p^{\nabla}(t,t_0)\right)$, and then
$e_p^{\nabla}(t,t_0)=\frac{p^{\rho}(t)}{1+p^{\rho}(t)\nu(t)}e_p(t,t_0)$, from
where it follows \eqref{eq:a}. Similarly, we have that
$\hat{e}_p^{\Delta}(t,t_0)=p^{\sigma}(t)\hat{e}_p^{\sigma}(t,t_0)
=p^{\sigma}(t)\left(\hat{e}_p(t,t_0)+\mu(t)\hat{e}_p^{\Delta}(t,t_0)\right)$,
and then
$\hat{e}_p^{\Delta}(t,t_0)=\frac{p^{\sigma}(t)}{1-p^{\sigma}(t)\mu(t)}\hat{e}_p(t,t_0)$,
from where \eqref{eq:b} holds.
\end{proof}

\begin{corollary}
\label{cor1} Let $t,t_0\in\T$, $p(t)\equiv p$, and  $1-\nu^2(t)p^2\neq 0$ for all $t\in\T$. Then, for $t\in\T^k_k$,

(a) $\sin_p^{\diam}(t,t_0)
=\frac{p}{1+\nu^2p^2}((1+\al\nu^2p^2)\cos_p(t,t_0)$
$+(1-\al)\nu p\sin_p(t,t_0))$;

(b) $\cos_p^{\diam}(t,t_0)
=\frac{-p}{1+\nu^2p^2}((1+\al\nu^2p^2)\sin_p(t,t_0)$
$-(1-\al)\nu p\cos_p(t,t_0))$;

(c) $\sinh_p^{\diam}(t,t_0)
=\frac{p}{1-\nu^2p^2}((1-\al\nu^2p^2)\cosh_p(t,t_0)$
$-(1-\al)\nu p\sinh_p(t,t_0))$;

(d) $\cosh_p^{\diam}(t,t_0)
=\frac{p}{1-\nu^2p^2}((1-\al\nu^2p^2)\sinh_p(t,t_0)$
$-(1-\al)\nu p\cosh_p(t,t_0))$.
\end{corollary}

\begin{corollary}
Let $t_0\in\T$, $p(t)\equiv p$, and
$1-\mu^2(t)p^2\neq 0$ for all $t\in\T$. Then, for $t\in\T^k_k$,

(a) $\widehat{\sin}_p^{\diam}(t,t_0)
=\frac{p}{1+\mu^2p^2}((1+(1-\al)\mu^2p^2)\widehat{\cos}_p(t,t_0)$
$ -\al\mu p\widehat{\sin}_p(t,t_0))$;

(b) $\widehat{\cos}_p^{\diam}(t,t_0)
=\frac{-p}{1+\mu^2p^2}((1+(1-\al)\mu^2p^2)\widehat{\sin}_p(t,t_0)$
$ +\al\mu p\widehat{\cos}_p(t,t_0))$;

(c) $\widehat{\sinh}_p^{\diam}(t,t_0)
=\frac{p}{1-\mu^2p^2}((1-(1-\al)\mu^2p^2)\widehat{\cosh}_p(t,t_0)$
$+\al\mu p\widehat{\sinh}_p(t,t_0))$;

(d) $\widehat{\cosh}_p^{\diam}(t,t_0)
=\frac{p}{1-\mu^2p^2}((1-(1-\al)\mu^2p^2)\widehat{\sinh}_p(t,t_0)$
$ +\al\mu p\widehat{\cosh}_p(t,t_0))$.
\end{corollary}


\section{Generalized monomials and polynomial series}

Let $\T$ be an arbitrary time scale. Let us define recursively
functions $h_k:\T\times\T\rightarrow\R$, $k\in\N\cup\{0\}$, as
follows:
 \begin{equation*}
    h_0(t,t_0)\equiv 1 \, , \quad
    h_{k+1}(t,t_0)=\int_{t_0}^th_k(\tau,t_0)\Delta \tau.
\end{equation*}

Similarly, we consider the monomials $\hat{h}_k(\cdot,t_0)$: they
are the functions $\hat{h}_k:\T\times\T\rightarrow\R$,
$k\in\N\cup\{0\}$, defined recursively by
 \begin{equation*}
    \hat{h}_0(t,t_0)\equiv 1 \, , \quad
    \hat{h}_{k+1}(t,t_0)=\int_{t_0}^t\hat{h}_k(\tau,t_0)\nabla \tau.
\end{equation*}
All functions $h_k$ are rd-continuous, all $\hat{h}_k$ are
ld-continuous. The derivatives of such functions show nice
properties: $h^{\Delta}_k(t,t_0) = h_{k-1}(t,t_0)$, $t\in\T^k$;
and $\hat{h}^{\nabla}_k(t,t_0)=\hat{h}_{k-1}(t,t_0)$, $t\in\T_k$,
where $t_0\in\T$ and derivatives are taken with respect to $t$. We
have that
\begin{equation*}
h_1(t,t_0)=\hat{h}_1(t,t_0)=t-t_0, \quad t, t_0\in\T \, .
\end{equation*}
For $\T=\R$, $h_k(t,t_0)=\hat{h}_k(t,t_0)=\frac{(t-t_0)^k}{k!}$.
Finding exact formulas of $h_k$ or $\hat{h}_k$ for an arbitrary
time scale is, however, not easy.
From \cite{HP} we have the following result:
\begin{theorem}
Let $t \in\T_k\cap\T^k$ and $t_0\in\T^{k^n}$. Then,
$\hat{h}_k(t,t_0)=(-1)^kh_k(t_0,t)$ for all $k\in\N\cup\{0\}$.
\end{theorem}

Next proposition gives explicit
formulas for homogenous time scales with $\mu(t)=c$, $c$ a
strictly positive constant. For that we need two notations of
factorial functions: for $k\geq 1$ we define $t^{\underline{k}} =
t(t-1)\cdots(t-k+1)$ and $t^{\overline{k}} = t(t+1)\cdots(t+k-1)$
with $t^{\underline{0}}:=1$ and $t^{\overline{0}}:=1$.

\begin{proposition} Let $c>0$ and $\T=c\Z$.
For $k\in\N\cup\{0\}$ the following equalities hold:

(a) $h_k(t,t_0)=c^k\binom{\frac{t-t_0}{c}}{k}
=c^k\frac{\left(\frac{t-t_0}{c}\right)^{\underline{k}}}{k!}$;

(b) $\hat{h}_k(t,t_0)
=c^k\frac{\left(\frac{t-t_0}{c}\right)^{\overline{k}}}{k!}$.
\end{proposition}

\begin{proof}
Firstly, $h_0(t,t_0)=\hat{h}_0(t,t_0)=1$. Next we observe that
$h_{k+1}^{\Delta}(t,t_0)=c^{k+1}\frac{\binom{\frac{\sigma(t)-t_0}{c}}{k+1}
-\binom{\frac{t-t_0}{c}}{k+1}}{c}\\
=c^{k}\left(
\binom{\frac{t-t_0}{c}+1}{k+1}-\binom{\frac{t-t_0}{c}}{k}\right)
=c^{k}\binom{\frac{t-t_0}{c}}{k}=h_{k}(t,t_0)$. Hence, by the
principle of mathematical induction, (a) holds for all
$k\in\N\cup\{0\}$. Since $\hat{h}_k(t,t_0)=(-1)^kh_k(t_0,t)$, (b)
is also true.
\end{proof}

\begin{remark}
Let $\T=c\Z$, $c> 0$. From the properties of factorial functions
it follows:
\begin{enumerate}
\item if $t\geq t_0\ \wedge \ k\geq \frac{t-t_0}{c}+1$, then
$h_k(t,t_0)=0$;

\item if $t\leq t_0\ \wedge \ k \geq \frac{|t-t_0|}{c}+1$, then
$\hat{h}_k(t,t_0)=0$.
\end{enumerate}
In particular, when $c=1$ and $\T=\Z$, we have:
\begin{enumerate}
\item if $t\geq t_0\ \wedge \ k> t-t_0$, then $h_k(t,t_0)=0$;
\item if $t\leq t_0\ \wedge \ k> |t-t_0|$, then
$\hat{h}_k(t,t_0)=0$.
\end{enumerate}
\end{remark}

In the next section we need the following results.

\begin{remark}
Let $\T=c\Z$, $c>0$, and $t\in\T$. Then, for $k\in\N\cup\{0\}$, the following holds:

 (a) $\frac{t^{\underline{k+1}}}{t^{\underline{k}}} =t-k, \ t\leq k$;

\bigskip

(b) $\frac{t^{\overline{k+1}}}{t^{\overline{k}}}=t+k, \ t\geq -k$.
\end{remark}

\begin{proposition}
\label{prop3012} Let $\T=c\Z$, $c>0$, $t,t_0\in\T$. Then,

(a) $\lim\limits_{k\rightarrow
\infty}\left|\frac{h_{k+1}(t,t_0)}{h_k(t,t_0)}\right|=c$ \ \ for
$t<t_0$;

(b) $\lim\limits_{k\rightarrow
\infty}\left|\frac{\hat{h}_{k+1}(t,t_0)}{\hat{h}_k(t,t_0)}\right|=c$\
\  for $t>t_0$.
\end{proposition}

\begin{proof}
Let $\T=c\Z$, $c>0$, and $t,t_0\in\T$. For $t<t_0$, it is enough
to notice that
$\frac{h_{k+1}(t,t_0)}{h_k(t,t_0)}=\frac{t-t_0-ck}{k+1}$ to prove
(a). Equality (b) is proved in a similar way: for $t>t_0$, we have:
$\frac{\hat{h}_{k+1}(t,t_0)}{\hat{h}_k(t,_0)}=\frac{t-t_0+ck}{k+1}$.
\end{proof}

\begin{remark}\cite{And}
Let $t_0\in\T^{k^n}$ and $k\in\N\cup\{0\}$. Then,
\begin{equation}\label{rel3}
    \hat{h}_{k+1}^{\nabla}(t_0,t)=-\sum_{j=0}^k\nu^j(t)\hat{h}_{k-j}(t_0,t)
\end{equation}
for $\T_k\cap\T^k$.
\end{remark}
As a consequence of (\ref{rel3}) and equalities $f^{\sigma}= f+\mu
f^{\Delta}$ and $f^{\rho} = f-\nu f^{\nabla}$, the following laws
of differentiation of generalized monomials follow.

\begin{corollary}
\label{derdor1}\mbox{}

(a) $\left[h_{k+1}(t,t_0)\right]^{\nabla} =
\sum_{j=0}^{k}(-1)^j\nu^j(t)h_{k-j}(t,t_0)$;

(b) $\left[\hat{h}_{k+1}(t,t_0)\right]^{\Delta}
=\sum_{j=0}^{k}\mu^j(t)\hat{h}_{k-j}(t,t_0)$.
\end{corollary}

\begin{example}\label{exdor1}
Let $\T$ be an homogenous time scale with
$\mu(t)=\nu(t)=c=const.$, $c\geq 0$. Let us recall that for $c=0$
we have $\T=\R$ and for $c=1$ we have $\T=\Z$. Then,
$\left[h_{k+1}(t,t_0)\right]^{\nabla}
= \sum_{j=0}^{k}(-1)^jc^jh_{k-j}(t,t_0)$,
$\left[\hat{h}_{k+1}(t,t_0)\right]^{\Delta}=
\sum_{j=0}^{k}c^j\hat{h}_{k-j}(t,t_0)$.

For $\T=\R$:
$\left[h_{k+1}(t,t_0)\right]^{\nabla}
=\left[\hat{h}_{k+1}(t,t_0)\right]^{\Delta} =\frac{(t-t_0)^k}{k!},$ \\
for $\T=\Z$:
$\left[h_{k+1}(t,t_0)\right]^{\nabla}=
\sum_{j=0}^{k}(-1)^jh_{k-j}(t,t_0)
=\binom{t-t_0}{0}-\binom{t-t_0}{1} +\cdots+(-1)^k\binom{t-t_0}{k},$
and $\left[\hat{h}_{k+1}(t,t_0)\right]^{\Delta}
=\sum_{j=0}^{k}\hat{h}_{k-j}(t,t_0)$.
\end{example}

\begin{proposition} Let $t,t_0\in\T$. Then,

(a) $h_1^{\diam}(t,t_0)=\hat{h}_1^{\diam}(t,t_0)\equiv 1$;

(b) $h_{k+1}^{\diam}(t,t_0)=h_k(t,t_0)+(1-\al)\sum_{j=1}^k(-1)^j\nu^j(t)h_{k-j}(t,t_0)$;

(c) $ \hat{h}_{k+1}^{\diam}(t,t_0)
=\hat{h}_k(t,t_0)+\al\sum_{j=1}^k\mu^j(t)\hat{h}_{k-j}(t,t_0)$.
\end{proposition}

\begin{proof}
From the definition of $\diam$-derivative and
Corollary~\ref{derdor1}, we have:

$h_{k+1}^{\diam}(t,t_0) =\al
h_{k+1}^{\Delta}(t,t_0)+(1-\al)h_{k+1}^{\nabla}(t,t_0) = \al
h_{k}(t,t_0)+(1-\al)\sum_{j=0}^{k}(-1)^j\nu^j(t)h_{k-j}(t,t_0)=
h_{k}(t,t_0)+(1-\al)\sum_{j=1}^{k}(-1)^j\nu^j(t)h_{k-j}(t,t_0)$.

Next, $\hat{h}_{k+1}^{\diam}(t,t_0)=\al
\hat{h}_{k+1}^{\Delta}(t,t_0)
+(1-\al)\hat{h}_{k+1}^{\nabla}(t,t_0)=
\al
\sum_{j=0}^{k}\mu^j(t)\hat{h}_{k-j}(t,t_0)+(1-\al)\hat{h}_{k}(t,t_0)=
\hat{h}_{k}(t,t_0)+\al\sum_{j=1}^{k}\mu^j(t)\hat{h}_{k-j}(t,t_0)$.
\end{proof}

\begin{theorem}\label{tayth1} {\rm \cite{HP}}
Assume that $f$ is $n+1$ delta-differentiable on  $\T^{k^{n+1}}$.
Let $t_0\in\T^{k^{n}}$, $t\in\T$. Then,
\begin{equation}
\label{tay1}
f(t)= \sum_{k=0}^{n} f^{\Delta^k}(t_0)h_k(t,t_0)+R_n(t,t_0),
\end{equation}
where $R_n(t,t_0)=\int_{t_0}^{t}
f^{\tr^{n+1}}(\tau)h_{n}(t,\sigma(\tau))\Delta\tau$.
\end{theorem}

\begin{theorem} \label{tayth2}{\rm \cite{And}} Assume that
$f$ is $n+1$ times nabla differentiable on $\T_{k^{n+1}}$. Let
$t_0\in\T_{k^{n}}$, $t\in\T$. Then,
\begin{equation}\label{tay2}
f(t)= \sum_{k=0}^{n}
f^{\nabla^k}(t_0)\hat{h}_k(t,t_0)
+\hat{R}_n(t,t_0),
\end{equation}
where $\hat{R}_n(t,t_0)=\int_{t_0}^{t}
f^{\nabla^{n+1}}(\tau)h_{n}(t,\rho(\tau))\nabla\tau$.
\end{theorem}

By a polynomial real series we usually understand a series of the
form $\Sigma_{k=0}^{\infty} a_kP_k(t)$, where $(P_k(t))_{n \in
\N}$ is a given sequence of polynomials in the variable $t$ and
$(a_k)_{k \in \N}$ is a given sequence of real numbers. In the
continuous case one has $P_k(t)=\frac{(t-t_0)^k}{k!}$. For the
time scales we are considering in this paper, we have
$P_k(t)=h_k(t,t_0)$ or $\hat{P}_k(t)=\hat{h}_k(t,t_0)$, and we
speak about \emph{generalized power series on time scales}
\cite{BG,MP2,MP3}.

\begin{definition}
\label{szereg} Let $\T$ be a time scale and let us fix $t_0\in\T$.
By a \emph{$\Delta$-polynomial series (on $\T$, originated at
$t_0$)} we shall mean the expression
$\sum_{k=0}^{\infty}a_kh_k(t,t_0)$, $t\in \T$; by a
\emph{$\nabla$-polynomial series (on $\T$, originated at $t_0$)}
we mean $\sum_{k=0}^{\infty}a_k\hat{h}_k(t,t_0)$, $t\in \T$, where
for each $k\in\N$, $a_k\in\R$. The sequence $(a_n)_{n\in\N}$ is
called the corresponding sequence of the series.
\end{definition}

\begin{remark}
For any fixed $t, t_0\in\T$, both type of series become ordinary
number series. If they are convergent for $t$ we say that the
polynomial series is convergent at $t$.
\end{remark}

If $\T=\Z$, then for each $t^*\in\Z$, $t^*\geq t_0$, the number
series $\sum_{k=0}^{\infty}a_kh_k(t^*,t_0)$ is convergent because
it is  finite. The same situation we have when $t^*\leq t_0$:
$\sum_{k=0}^{\infty}a_k\hat{h}_k(t^*,t_0)$ is finite, so
convergent.

In \cite{BG} and \cite{MP2} it is proved the following:
\begin{proposition}
Let $t_0\in\T$. If the power series
$\sum\limits_{k=0}^{\infty}a_k\frac{(t-t_0)^k}{k!}$ with the
corresponding sequence of coefficients $(a_k)_{k\in \N_0}$ is
convergent at $t^*\in\T$ and $t^*\geq t_0$, then the polynomial
series is convergent for all values of $t\in\T$ such that
$t_0<t<t^*\in\T$.
\end{proposition}

Two polynomial series of the same type can be added
and multiplied by scalars giving the same type of series. We can
define the $\Delta$-derivative of $\Delta$-polynomial series:
$\left(\sum_{k=0}^{\infty}a_{k}h_k(t,t_0)\right)^{\Delta}=
\sum_{k=0}^{\infty}a_{k+1}h_k(t,t_0)$. Similarly, we have the
$\nabla$-derivative of $\nabla$-polynomial series in the form
$\left(\sum_{k=0}^{\infty}a_{k}\hat{h}_k(t,t_0)\right)^{\nabla}=
\sum_{k=0}^{\infty}a_{k+1}\hat{h}_k(t,t_0)$. Additionally, if the
$\Delta$-polynomial series is convergent for $t\in\langle
t_0,t^*)\cap\T$ and if the $\nabla$-polynomial series is
convergent for $t\in (t^*,t_0\rangle\cap\T$, then their
derivatives are also convergent on the same sets. From
Corollary~\ref{derdor1} we obtain the following result.

\begin{proposition}
\label{derser} Let $ t_0\in\T$, $M>0$, and $\left(a_k\right)_{k\in\N\cup\{0\}}$ be a
sequence such that $|a_k|\leq M^k$ for each $k$. We have:

(a) Let $I=\{t\in\T: \rho(t)\geq t_0 \ \wedge \ \nu(t)M<1\}$. Then, the series $\sum_{k=0}^{\infty}a_{k}h_k(t,t_0)$ is
convergent for $t\in I$, and
$\left(\sum_{k=0}^{\infty}a_{k}h_k(t,t_0)\right)^{\nabla}
=\sum_{k=0}^{\infty}\left(\sum_{j=0}^{\infty}
(-1)^j\nu^j(t)a_{j+k+1}\right)h_k(t,t_0)$
exists and it is convergent for $t\in I$.

(b) Let $J=\{t\in\T: \sigma(t)\leq t_0 \ \wedge \ \mu(t)M<1\}$. Then, the series $\sum_{k=0}^{\infty}a_{k}\hat{h}_k(t,t_0)$ is
convergent for $t\in J$, and
$\left(\sum_{k=0}^{\infty}a_{k}\hat{h}_k(t,t_0)\right)^{\Delta}
=\sum_{k=0}^{\infty}\left(\sum_{j=0}^{\infty}
\mu^j(t)a_{j+k+1}\right)\hat{h}_k(t,t_0)$
exists and it is convergent for $t\in J$.
\end{proposition}

\begin{remark}
Let $\T=c\Z$, $c>0$. There is no problem with convergence (i) in
points $t\geq t_0$ for series of the first type, (ii) at points
$t\leq t_0$ for series of the second (``hat'') type, because such
series are finite.
\end{remark}

In \cite{BG} and \cite{MP2} one can find generalized series for an
exponential $e_p(t,t_0)$ with constant function $p(t)\equiv p$:
for $t\in\T$ and $t\geq t_0$ one has
$e_p(t,t_0)=\sum_{k=0}^{\infty}p^kh_k(t,t_0).$
It follows that $e_p^{\Delta}(t,t_0)
=p\sum_{k=0}^{\infty}p^kh_k(t,t_0)$, which gives the rule
$e_p^{\Delta}(t,t_0)=pe_p(t,t_0)$.

As in Proposition~\ref{diamexpder}, let us consider now a time
scale with $\sigma(\rho(t))=t$. Then,
\begin{equation}
e_p^{\nabla}(t,t_0)
=\sum_{k=0}^{\infty}p^{k+1}\left(\sum_{j=0}^{\infty}
(-1)^j\nu^j(t)p^{j}\right)h_k(t,t_0) \, ,
\end{equation}
where
$\sum_{j=0}^{\infty}(-1)^j\nu^j(t)p^{j}=\frac{1}{1+p\nu(t)}$ if
$|p\nu(t)|<1$. This gives that
\begin{equation}
e_p^{\nabla}(t,t_0)=\frac{p}{1+p\nu(t)}\sum_{k=0}^{\infty}
p^kh_k(t,t_0)=\frac{p}{1+p\nu(t)}e_p(t,t_0)\, ,
\end{equation}
and then $e_p^{\diam}(t,t_0)$ is as in
Proposition~\ref{diamexpder}.

In \cite{MP3} it is proved that
$\hat{e}_p(t,t_0)=\sum_{k=0}^{\infty}p^k\hat{h}_k(t,t_0)$ for
$t\geq t_0$. Then, we have that
$\hat{e}_p^{\nabla}(t,t_0)=p\hat{e}_p(t,t_0)$. We obtain that the
$\Delta$-derivative of $\hat{e}_p(\cdot, t_0)$ with respect to $t$ is given by the formula
$\hat{e}_p^{\Delta}(t,t_0)=\sum_{k=0}^{\infty}
p^{k+1}\left(\sum_{j=0}^{\infty}\mu^j(t)p^{j}\right)\hat{h}_k(t,t_0)$,
where $\sum_{j=0}^{\infty}\mu^j(t)p^{j}=\frac{1}{1-p\mu(t)}$ if
$|p\mu(t)|<1$. This gives that
$\hat{e}_p^{\Delta}(t,t_0)=\frac{p}{1-p\mu(t)}
\sum_{k=0}^{\infty}p^k\hat{h}_k(t,t_0)
=\frac{p}{1-p\mu(t)}\hat{e}_p(t,t_0)$
and then $\hat{e}_p^{\diam}(t,t_0)$ is also as in
Proposition~\ref{diamexpder}.


\section{Combined series}

The diamond-$\alpha$ derivative reduces to the standard $\Delta$
derivative for $\alpha=1$ and to the standard $\nabla$ derivative
for $\alpha=0$. The same ``weighted'' type definition is proposed
for the diamond-$\alpha$ integral. Based on this simple idea, we
introduce diamond type monomials.
Let us begin with the trivial remark that for any $f:\T\rightarrow
\R$ we can write $f(t)=\alpha f(t)+(1-\alpha)f(t)$.

\begin{theorem}
Assume that $f$ is $n+1$ delta- and nabla-differentiable on
$\T^{k^{n+1}}$ and $\T_{k^{n+1}}$, respectively. Let
$t_0\in\T_{k^n}\cap\T^{k^{n}}$, $t\in\T$. Then,
\begin{equation*}
f(t)= \alpha S_n(t,t_0)+(1-\alpha)\hat{S}_n(t,t_0)+
\tilde{R}_n(t,t_0),
\end{equation*}
where
$S_n(t,t_0)=\sum_{k=0}^{n} f^{\Delta^k}(t_0)h_k(t,t_0) \, ,
\hat{S}_n(t,t_0) = \sum_{k=0}^{n}
f^{\nabla^k}(t_0)\hat{h}_k(t,t_0) \, $, and
$\tilde{R}_n(t,t_0)=\alpha R_n(t,t_0) + (1-\alpha)\hat{R}_n(t,t_0)
\, ,$
with remainders $R_n(t,t_0)$ and $\hat{R}_n(t,t_0)$ given as in
Theorems~\ref{tayth1} and \ref{tayth2}.
\end{theorem}

\begin{definition}\label{diam_ser}
Let $\T$ be a time scale and $t_0\in\T$. By a
\emph{combined-polynomial series (on $\T$, originated at $t_0$)}
we shall mean the expression
\begin{equation}
\label{diam_ser1}
S_{\al}(t,t_0):=\al\sum_{k=0}^{\infty}a_kh_k(t,t_0)
+(1-\al)\sum_{k=0}^{\infty}b_k\hat{h}_k(t,t_0),
\end{equation}
where $t\in\T$ and $\al\in[0,1]$.
\end{definition}

\begin{remark}
If in (\ref{diam_ser1}) we put $\al=1$, then we have a
$\Delta$-polynomial series. For $\al=0$ we obtain
$\nabla$-polynomial series. A combined-series is convergent if
both types of polynomial series are convergent. For fixed
$t,t_0\in\T$ we get usual number series, so we can say that the
series originated at $t_0$ is convergent at $t$ if it is
convergent as a number series.
\end{remark}

\begin{proposition}
Let $\T=c\Z$, $c>0$, and $(a_0,a_1,\ldots)$, $(b_0,b_1,\ldots)$ be
two real sequences with nonzero elements such that
$\lim\limits_{k\rightarrow\infty}\left|\frac{a_{k+1}}{a_k}\right|
<\frac{1}{c}$,
$\lim\limits_{k\rightarrow\infty}\left|\frac{b_{k+1}}{b_k}\right|
<\frac{1}{c}$. Then, the combined-polynomial series
$$S_{\al}(t,t_0)=\al\sum_{k=0}^{\infty}a_kh_k(t,t_0)
+(1-\al)\sum_{k=0}^{\infty}b_k\hat{h}_k(t,t_0) \, ,$$
$\al\in[0,1]$, is convergent for all $t\in\T$.
\end{proposition}

\begin{proof}
Based on Proposition \ref{prop3012}, we consider: $t=t_0$, when
combined-series $S_{\al}(t_0,t_0)=\al a_0+(1-\al)b_0$; $t>t_0$, when
the first part is finite the second is convergent; $t<t_0$, when
we have opposite situation to the previous one.
\end{proof}

\begin{example}
Let $\T=\Z$ and $f(t)=2^t$. Then, $f^{\Delta^k}(t)=2^t$ and
$f^{\Delta^k}(0)=1$ for $k\in\N\cup\{0\}$. Additionally,
$2^t=\sum_{k=0}^{\infty}h_k(t,0)=\sum_{k=0}^t\binom{t}{k}$ for any
$t\geq 0$. But this series is not convergent for $t<0$. We have
$f^{\nabla^k}(t)=2^{t-k}$ and $f^{\nabla^k}(0)=2^{-k}$. The series
$\sum_{k=0}^{\infty}\frac{1}{2^k}\hat{h}_k(t,0)$ is convergent for
any $t\in\Z$. For that let
$c_k(t)=\frac{1}{2^k}\hat{h}_k(t,0)=\frac{t^{\bar{k}}}{2^kk!}$.
Then, $$\lim\limits_{k\rightarrow\infty}\frac{c_{k+1}(t)}{c_k(t)}
=\lim\limits_{k\rightarrow\infty}
\frac{(t+1)(t+2)\cdots(t+k)}{2^{k+1}(k+1)!}
\cdot\frac{2^kk!}{t(t+1)(t+2)\cdots(t+k-1)}
=\lim\limits_{k\rightarrow\infty}\frac{k+t}{2(k+1)t}=\frac{1}{2}$$
for each fixed $t>0$. The combined-polynomial series has the form $\al\sum_{k=0}^{\infty}\frac{t^{\underline{k}}}{k!}
+(1-\al)\sum_{k=0}^{\infty}\frac{1}{2^k}\frac{t^{\overline{k}}}{k!}$
and is convergent for $t\geq 0$.
\end{example}


\section{Conclusions}

Polynomial series have been used in the literature for solving a
variety of problems in control. In this paper we define Taylor
series via diamond-alpha derivatives on time scales and provide
the first steps on the correspondent theory. Such a theory
provides a general framework that is valid for discrete,
continuous or hybrid series. We trust that the polynomial series
here introduced are important in the analysis of control systems
on time scales.


\section*{Acknowledgments}

The first author was supported by
Bia{\l}ystok Technical University grant S/WI/1/07;
the second author by the R\&D unit CEOC,
via FCT and the EC fund FEDER/POCI 2010.



\end{document}